\documentclass[11pt]{article}
\usepackage{amsmath, amsthm}
\usepackage{latexsym}
\usepackage{graphicx}
\usepackage[scanall]{psfrag}
\makeatletter
\newfont{\footsc}{cmcsc10 at 8truept}
\newfont{\footbf}{cmbx10 at 8truept}
\newfont{\footrm}{cmr10 at 10truept}
\makeatother
\newtheorem{theorem}{\bf Theorem}
\newtheorem{proposition}{\bf Proposition}
\newtheorem{lemma}{\bf Lemma}
\newtheorem{corollary}{\bf Corollary}

\begin{document}
\title{Prior Ordering and Monotonicity in Dirichlet Bandits}

\author{Yaming Yu\\
\small Department of Statistics\\[-0.8ex]
\small University of California\\[-0.8ex] 
\small Irvine, CA 92697, USA\\[-0.8ex]
\small \texttt{yamingy@uci.edu}}

\date{}
\maketitle

\begin{abstract}
One of two independent stochastic processes (arms) are to be selected at each of $n$ stages.  The selection is sequential and depends on past observations as well as the prior information.  Observations from arm $i$ are independent given a distribution $P_i$, and, following Clayton and Berry (1985), $P_i$'s have independent Dirichlet process priors.  The objective is to maximize the expected future-discounted sum of the $n$ observations.  We study structural properties of the bandit, in particular how the maximum expected payoff and the optimal strategy vary with the Dirichlet process priors.  The main results are (i) for a particular arm and a fixed prior weight, the maximum expected payoff increases as the mean of the Dirichlet process prior becomes larger in the increasing convex order; (ii) for a fixed prior mean, the maximum expected payoff decreases as the prior weight increases.  Specializing to the one-armed bandit, the second result captures the intuition that, given the same immediate payoff, the more is known about an arm, the less desirable it becomes because there is less to learn when selecting that arm.  This extends some results of Gittins and Wang (1992) on Bernoulli bandits and settles a conjecture of Clayton and Berry (1985). 

{\bf Keywords:} convex order; Dirichlet bandits; sequential decision; two-armed bandits.

{\bf MSC 2010:} Primary 62L05, 62C10; Secondary 62L15, 60E15. 
 
\end{abstract}

\section{Introduction}
Bandit problems are classical problems in statistical decision theory and have received considerable attention; see Berry and Fristedt (1985) for an overview.  We consider discrete-time, finite-horizon, two-armed bandits from a Bayesian perspective.  At each of $n$ stages, an observation is taken from one of two stochastic processes (arms).  A {\it strategy} specifies which process to select based on past observations.  The objective is to maximize the expected payoff, $\sum_{i=1}^n a_i Z_i$, where $Z_i$ is the observation at stage $i$ and $A_n\equiv (a_1, a_2,\ldots, a_n)$ is a  discount sequence satisfying $a_i\geq 0$ and $\sum_{i=1}^n a_i >0$.  A strategy is optimal if it achieves the maximum expected payoff.  An arm is optimal initially if there exists an optimal strategy that selects that arm at the first stage. 

The most widely studied bandit problem is the Bernoulli bandit, where each arm generates a sequence of exchangeable Bernoulli random variables.  Bernoulli bandits are important as a model for clinical trials.  Others such as normal bandits have also been extensively studied (Chernoff 1968).  Extending the Bernoulli bandit, Clayton and Berry (1985) have introduced a one-armed Bayesian nonparametric bandit using Dirichlet process priors (Ferguson 1973).  Chattopadhyay (1994) extends this and studies the two armed Dirichlet bandit, which is also the setting of this work.  Associated with arms $1$ and $2$ are probability measures $P_i,\ i=1,2$, respectively.  Observations from arm $i$ are independent samples given $P_i$; observations from different arms are independent.  The $P_i$'s themselves are treated as random, with independent Dirichlet process priors.  Specifically, $P_i\sim \rm{DP}(\alpha_i)$, where $\alpha_i$ is a finite nonnull measure with a finite first moment.  It is often helpful to write $\alpha_i=M_i F_i$ where $M_i=\alpha_i(\mathbf{R})$ so that $F_i$ is a probability distribution.  We refer to $F_i$ and $M_i$ as the prior mean distribution and prior weight of the Dirichlet process, respectively.  We use $(\alpha_1, \alpha_2; A_n)$ to denote such a Dirichlet bandit with discount sequence $A_n$. 

For such problems one must balance the desire to maximize the immediate payoff and the need to explore a less known arm in the hope of higher payoff later on (the exploitation versus exploration dilemma).  Optimal strategies are usually  specified through backward induction and are nontrivial to compute.  Nevertheless certain structural properties such as the stay-on-a-winner rule (Bradt, Johnson and Karlin 1956; Berry 1972) often hold under suitable conditions.  For Dirichlet bandits with known arm 2, Clayton and Berry (1985) obtain several structural results.  In particular, the maximum expected payoff increases as $F_1$, the mean of the Dirichlet process prior for arm 1, increases in the usual stochastic order.  Also, a version of the stay-on-a-winner rule holds: if arm 1 is optimal initially then it is optimal at the next stage provided that the initial observation from arm 1 is sufficiently large.  Such results have been extended to the general two-armed Dirichlet bandits (Chattopadhyay 1994). 

This paper studies further structural properties of Dirichlet bandits, in particular how the value of the bandit (i.e., the maximum expected payoff) varies with the Dirichlet process priors.  The main results are (i) the value increases as the mean of the Dirichlet process for any arm becomes larger in the increasing convex order (defined below); (ii) the value decreases as the prior weight of the Dirichlet process of an arm increases.  The second result agrees with the intuition that, given the same immediate payoff, an arm is less appealing when more is known about it, because there remains less to be explored.  Though easy to state and intuitively appealing, such results are often difficult to prove.  We mention a long-standing conjecture of Berry (1972), which states that for a finite-horizon Bernoulli two-armed bandit with uniform discounting and independent ${\rm Beta}(u_i, v_i)$ priors, $i=1,2,$ for arms 1 and 2 respectively, if $u_1/v_1=u_2/v_2$ and $u_1+v_1<u_2+v_2$, then arm 1 is preferred to arm 2 at the initial pull.  If, instead of finite-horizon uniform discounting, we assume infinite-horizon geometric discounting, then the corresponding conjecture is true, as shown by Gittins and Wang (1992), who also prove analogous results for some other parametric bandits.  Geometric discounting is special in that the optimal strategy for a multi-armed bandit is characterized by a ``dynamic allocation index,'' or Gittins index (Gittins and Jones 1974; Gittins 1979; Whittle 1980), which reduces the problem to several one-armed bandits. 

As the Bernoulli bandit is a special case of the Dirichlet bandit, our results may be regarded as a generalization of Gittins and Wang (1992), although our method of proof, based on convexity and stochastic orders, is different.  
Our main result (Corollary~\ref{coro2}) confirms a conjecture of Clayton and Berry (1985) concerning the break-even value in the one-armed Dirichlet bandit.  We also prove another conjecture of Clayton and Berry (1985) concerning the break-even observation when both arms are optimal initially (Proposition~\ref{prop1}).  These results will hopefully shed some light on the conjecture of Berry (1972).  See Herschkorn (1997) for related results and conjectures on the Bernoulli bandit. 

We find the usual stochastic order, the convex order and the increasing convex order particularly helpful in formulating and deriving the main results.  For random variables $Z_1$ and $Z_2$ taking values on $\mathbf{R}$, we write $Z_1\leq_{\rm st} Z_2$ (respectively, $Z_1\leq_{\rm cx} Z_2$), if 
\begin{equation}
\label{orders}
E\phi(Z_1)\leq E\phi(Z_2)
\end{equation} 
for every increasing (respectively, convex) function $\phi$ such that the expectations exist.  If $Z_1\leq_{\rm st} Z_2$ then we also say $Z_2$ is to the right of $Z_1$.  We say $Z_1$ is smaller than $Z_2$ in the increasing convex order, written as $Z_1\leq_{\rm icx} Z_2,$ if (\ref{orders}) holds for every increasing and convex function $\phi$ such that the expectations exist.  Hence $\leq_{\rm icx}$ is implied by either $\leq_{\rm st}$ or $\leq_{\rm cx}$.  The convex order is concerned with variability.  For example, if $Z_1\leq_{\rm cx} Z_2$, both with finite second moments, then $EZ_1=EZ_2$ and $Var(Z_1)\leq Var(Z_2)$.  Another basic property is closure under mixtures: if distributions $F_i,\ G_i,\ i=1,2,$ satisfy $F_1\leq_{\rm cx} F_2$ and $G_1\leq_{\rm cx} G_2$ then $\rho F_1 +(1-\rho) G_1\leq_{\rm cx} \rho F_2 + (1-\rho) G_2,\ \rho\in [0, 1]$; closure under mixtures also holds for $\leq_{\rm icx}$ and $\leq_{\rm st}$.  (We use the notation $\leq_{\rm st},\ \leq_{\rm cx},\ \leq_{\rm icx}$ with distribution functions as well as random variables.)  For further properties and applications of various stochastic orders, see M\"{u}ller and Stoyan (2002) and Shaked and Shanthikumar (2007). 

\section{Prior mean monotonicity}
Let us denote the maximum expected payoff of a two-armed Dirichlet bandit $(\alpha_1, \alpha_2; A_n)$ by $W(\alpha_1, \alpha_2; A_n)$.  Let $W^i(\alpha_1, \alpha_2; A_n)$ be the expected payoff when selecting arm $i$ initially and using an optimal strategy thereafter.  Then 
\begin{equation}
\label{Wdef1}
W(\alpha_1, \alpha_2; A_n)=\max\left\{W^1(\alpha_1, \alpha_2; A_n), W^2(\alpha_1, \alpha_2; A_n)\right\}. 
\end{equation}
Suppose arm 1 is selected initially, resulting in an observation $X$.  Because the prior on $P_1$ is a Dirichlet process, the posterior is again a Dirichlet process ${\rm DP}(\alpha_1+\delta_X)$, where $\delta_x$ denotes a point mass at $x$.  Thus we have 
\begin{align}
\label{Wdef2}
W^1(\alpha_1, \alpha_2; A_n) &= a_1 \mu_1 + \left. E\left[W(\alpha_1 +\delta_X, \alpha_2; A^1_n)\right|\alpha_1\right],\\
\label{Wdef3}
W^2(\alpha_1, \alpha_2; A_n) &= a_1 \mu_2 + \left. E\left[W(\alpha_1, \alpha_2 +\delta_Y; A^1_n)\right|\alpha_2\right],
\end{align}
where $A_n^1=(a_2, a_3, \ldots, a_n)$ and $\mu_i$ denotes the first moment of $\alpha_i$, which is also the expected value of an observation from arm $i$.  In $E[g(X)|\alpha]$, the distribution of $X$ is $\alpha/M$ with $M=\alpha(\mathbf{R})$.  The quantities $W,\ W^1$ and $W^2$ are well defined and finite as long as $\alpha_i,\ i=1,2,$ have finite first moments, which we assume throughout. 

Lemma \ref{lem1} reveals a convexity property of $W$ which we shall use repeatedly.

\begin{lemma}
\label{lem1}
Let $\alpha$ be a finite measure on $\mathbf{R}$ with a finite mean.  Then, for $u, v\in \mathbf{R}$ and $r>0$, the function $W(\alpha +\rho \delta_u + (r-\rho)\delta_v, \alpha_2; A_n)$ is convex in $\rho\in [0,r]$.  
\end{lemma}
\begin{proof}
Let us use induction on $n$.  It is easy to check that the claim holds for $n=1$.  For $n\geq 2$, we note that by (\ref{Wdef1}) it suffices to show that each of $W^i(\alpha +\rho \delta_u + (r-\rho) \delta_v, \alpha_2; A_n),\ i=1,2,$ is convex in $\rho\in [0,r]$.  Since the mean of $\alpha +\rho \delta_u + (r-\rho) \delta_v$ is linear in $\rho$, by (\ref{Wdef2}) and (\ref{Wdef3}), we only need to show that both 
\begin{equation}
\label{convex1}
\left. E\left[W(\alpha + \rho \delta_u +(r-\rho)\delta_v + \delta_X, \alpha_2; A^1_n)\right|\alpha +\rho \delta_u + (r-\rho)\delta_v\right]\quad {\rm and}
\end{equation}
\begin{equation}
\label{convex2}
\left. E\left[W(\alpha +\rho\delta_u +(r-\rho)\delta_v, \alpha_2 +\delta_Y; A^1_n)\right|\alpha_2\right]
\end{equation}
are convex in $\rho$.  Convexity of (\ref{convex2}) follows from the induction hypothesis.  To deal with (\ref{convex1}), we directly compute
\begin{align}
\nonumber
E[W(\alpha &+ \rho \delta_u +(r-\rho)\delta_v + \delta_X, \alpha_2; A^1_n )|\alpha +\rho \delta_u + (r-\rho)\delta_v]\\ 
\label{convex3}
= &\frac{M}{M + r} E[W(\alpha +\rho \delta_u + (r-\rho)\delta_v +\delta_X, \alpha_2; A^1_n)|\alpha]\\
\label{convex4}
&\quad +\frac{\rho \phi(\rho+1) + (r-\rho) \phi(\rho)}{M + r},
\end{align}
where $M=\alpha(\mathbf{R})$ and 
$$\phi(\rho)= W(\alpha + \rho \delta_u + (r+1-\rho)\delta_v, \alpha_2; A^1_n).$$  
By the induction hypothesis, $\phi(\rho)$ is convex in $\rho\in [0, r+1]$.  We claim that this implies that $\psi(\rho)\equiv \rho \phi(\rho+1) + (r-\rho) \phi(\rho)$ is convex in $\rho\in [0, r]$.  In fact, if $\phi(\rho)$ is twice differentiable, then we have
$$\psi''(\rho) = 2(\phi'(\rho+1)-\phi'(\rho)) + \rho \phi''(\rho+1) + (r-\rho) \phi''(\rho)\geq 0,\quad \rho\in [0, r],$$ 
by the convexity of $\phi$.  A standard limiting argument shows that $\psi(\rho)$ is convex in $\rho \in [0, r]$ as long as $\phi(\rho)$ is convex in $\rho\in [0, r+1]$ without assuming differentiability.  Hence the second term (\ref{convex4}) is convex.  The first term (\ref{convex3}) is convex in $\rho\in [0, r]$ by the induction hypothesis, since in this expectation $X$ is distributed according to $\alpha/M$ independently of $\rho$.  Thus the convexity of (\ref{convex1}) is established.  
\end{proof}

Theorem \ref{thm1} says that the value of the bandit increases as the mean of the Dirichlet process prior for any arm becomes stochastically larger and more dispersed.  This strengthens Proposition 2.2 of Clayton and Berry (1985) who consider the usual stochastic order rather than the increasing convex order. 

\begin{theorem}
\label{thm1}
If $M>0$ and $F\leq_{\rm icx} \tilde{F}$, both with finite means, then 
$$W(M F, \alpha_2; A_n)\leq W(M \tilde{F}, \alpha_2; A_n).$$ 
\end{theorem}

\begin{proof}
Let us use induction.  The claim obviously holds for $n=1$.  For $n\geq 2$ we have $W^2(M F, \alpha_2; A_n) \leq W^2(M\tilde{F}, \alpha_2; A_n)$ by (\ref{Wdef3}) and the induction hypothesis.  Moreover,  
\begin{align*}
W^1(M F, \alpha_2; A_n) &=a_1 E(X|F) + E[W(M F +\delta_X, \alpha_2; A^1_n)|F]\\
&\leq a_1 E(X|\tilde{F}) + E[W(M \tilde{F} + \delta_X, \alpha_2; A^1_n)|F]\\
&\leq a_1 E(X|\tilde{F}) + E[W(M \tilde{F} + \delta_X, \alpha_2; A^1_n)|\tilde{F}]\\
&=W^1(M \tilde{F}, \alpha_2; A_n), 
\end{align*}
where the first inequality follows from $F\leq_{\rm icx} \tilde{F}$ and the induction hypothesis, noting that $(M F +\delta_x)/(M+1) \leq_{\rm icx} (M \tilde{F} + \delta_x)/(M+1)$ for any $x$; the second inequality holds by the definition of $\leq_{\rm icx}$, because $W(M\tilde{F} +\delta_x, \alpha_2; A^1_n)$ is an increasing, convex function of $x$.  To show this, fix $-\infty<u<v<\infty$.  It is easy to show $(M\tilde{F} +\delta_u)/(M+1) \leq_{\rm icx} (M\tilde{F} +\delta_v)/(M+1),$ which, by the induction hypothesis, implies $W(M\tilde{F} +\delta_u, \alpha_2; A^1_n) \leq W(M\tilde{F} +\delta_v, \alpha_2; A^1_n)$.  Moreover,
\begin{align*}
W(M\tilde{F} +\delta_u, \alpha_2; A^1_n) &+ W(M\tilde{F} +\delta_v, \alpha_2; A^1_n)\\
&\geq 2 W(M\tilde{F} + (\delta_u +\delta_v)/2, \alpha_2; A^1_n)\\
&\geq 2 W(M\tilde{F} + \delta_{(u+v)/2}, \alpha_2; A^1_n),
\end{align*}
where the first inequality follows from Lemma \ref{lem1}, and the second inequality holds by the induction hypothesis, noting that 
$$\frac{M\tilde{F} + \delta_{(u+v)/2}}{M+1} \leq_{\rm icx} \frac{M\tilde{F} + (\delta_u +\delta_v)/2}{M+1}.$$  
Hence $W(M\tilde{F} +\delta_x, \alpha_2; A^1_n)$ is convex in $x$ as needed. 
\end{proof}

{\bf Remark 1.}  Theorem \ref{thm1} extends to bandits with more than two arms.  That is, the maximum expected payoff increases when the mean of the Dirichlet process prior for any arm becomes larger in the increasing convex order.  We present the two-armed version for notational convenience.  The discount sequence in Theorem~\ref{thm1} is very general, i.e., we only assume $A_n$ is nonnegative.  By approximation, this can be further extended to the infinite-horizon case assuming $\sum_{i=1}^\infty a_i <\infty$.  Similar comments apply to Theorem~\ref{thm2} in Section~3. 

When arm 2 has a known distribution $P_2$ with mean $\lambda$, the problem reduces to a one-armed bandit.  Without loss of generality we may assume the known arm yields a constant payoff $\lambda$ at each stage, i.e., we consider the  $(\alpha, \delta_{\lambda}; A_n)$ bandit (the subscript on $\alpha_1$ is dropped for convenience).  It is well known that, assuming the discount sequence is regular in the sense that $(\sum_{i\geq j+1} a_i)^2 \geq (\sum_{i\geq j} a_i) (\sum_{i\geq j+2} a_i)$ for all $j\geq 1$, this one-armed bandit is an optimal stopping problem, i.e., if at any stage it is optimal to pull arm 2 then arm 2 should be used in all subsequent stages; see Berry and Fristedt (1979).  If $A_n$ is regular, then there exists a break-even value $\Lambda(\alpha; A_n)$ for the $(\alpha, \delta_\lambda; A_n)$ bandit, such that arm 1 is optimal initially if and only if $\lambda \leq \Lambda(\alpha; A_n)$ and arm 2 is optimal initially if and only if $\lambda \geq \Lambda(\alpha; A_n)$.  For infinite-horizon geometric discounting, this break-even value is also known as the dynamic allocation index or Gittins index (Gittins and Jones 1974).  The following result holds by the optimal stopping characterization and is stated for uniform discounting as Lemma 2.1 in Clayton and Berry (1985). 

\begin{lemma}
\label{lemlam}
If $A_n$ is regular, then $\Lambda(\alpha; A_n)$ is the smallest $\lambda$ such that $W(\alpha, \delta_\lambda; A_n)\leq \lambda \sum_{i=1}^n a_i $. 
\end{lemma}

Lemma \ref{lemlam} and Theorem \ref{thm1} yield the following result comparing $\Lambda(\alpha; A_n)$.

\begin{corollary}
\label{coro1}
For $M>0$ and $F\leq_{\rm icx} \tilde{F}$, both with finite means, we have $\Lambda(MF; A_n)\leq \Lambda(M \tilde{F}; A_n)$, assuming $A_n$ is a regular discount sequence.
\end{corollary}

Suppose $A_n$ is regular.  Monotonicity and continuity considerations (see Clayton and Berry 1985) show that, for the $(\alpha, \delta_\lambda; A_n)$ bandit there exists a break-even observation $b(\alpha; A_n)$ such that if both arms are optimal initially, and an observation $x$ is taken from arm 1, then arm 1 remains optimal if $x\geq b(\alpha; A_n)$ and arm 2 becomes optimal if $x\leq b(\alpha; A_n)$.  That is, 
\begin{align*}
\Lambda(\alpha; A_n) &\geq \Lambda(\alpha+\delta_x; A_n^1),\quad {\rm if}\ x\leq b(\alpha; A_n);\\
\Lambda(\alpha; A_n) &\leq \Lambda(\alpha+\delta_x; A_n^1),\quad {\rm if}\ x\geq b(\alpha; A_n).
\end{align*}
Calculating this break-even observation is nontrivial.  In the case of uniform discounting, Clayton and Berry (1985) prove an upper bound for $b(\alpha; A_n)$ and conjecture that $b(\alpha; A_n)\geq \Lambda(\alpha; A_n)$ based on numerical evidence.  We confirm this in Proposition \ref{prop1}. 

\begin{proposition}
\label{prop1}
Suppose $n\geq 2$ and $A_n$ is regular and all positive.  Then $b(\alpha; A_n)\geq \Lambda(\alpha; A_n)$. 
\end{proposition}
As noted by Berry and Fristedt (1985; p.\ 131), Proposition~\ref{prop1} has an intuitive interpretation.  Suppose both arms are optimal initially, and arm 1 is selected.  If the initial pull on arm 1 yields no more than $\Lambda(\alpha; A_n)$, which is the yield of arm 2 per pull, the hope of getting higher payoff fades.  Not surprisingly, arm 2 becomes optimal afterwards.  This suggests that the break-even observation is at least $\Lambda(\alpha; A_n)$. 

To prove Proposition~\ref{prop1} we need a lemma.
\begin{lemma}
\label{lem3}
For $c>0,\ \lambda\in \mathbf{R}$ and an arbitrary discount sequence $A_n$, we have  
$$W(\alpha+ c\delta_\lambda, \delta_\lambda; A_n)\leq W(\alpha, \delta_\lambda; A_n).$$ 
\end{lemma}
\begin{proof}
We use induction on $n$.  The $n=1$ case is easy.  Suppose $n\geq 2$.  Let us write $M=\alpha(\mathbf{R})$ and let $\mu$ be the first moment of $\alpha$.  Direct calculation using (\ref{Wdef1})--(\ref{Wdef3}) yields  
\begin{align}
\label{lam_mono}
W(\alpha + c\delta_\lambda, \delta_\lambda; A_n) =\max\left\{\frac{M\phi_0 + c\phi_1}{M+c},\ \phi_2\right\},
\end{align}
where
\begin{align*}
\phi_0 &= a_1 \mu + E\left[ W(\alpha+ c\delta_\lambda +\delta_X, \delta_\lambda; A_n^1)|\alpha\right];\\
\phi_1 &= a_1 \lambda + W(\alpha + (c+1)\delta_\lambda, \delta_\lambda; A_n^1);\\
\phi_2 &= a_1\lambda + W(\alpha + c\delta_\lambda, \delta_\lambda; A_n^1). 
\end{align*}
Applying the induction hypothesis, and then (\ref{Wdef1}) and (\ref{Wdef2}), we get 
\begin{align*}
\phi_0 &\leq a_1\mu + E\left[ W(\alpha +\delta_X, \delta_\lambda; A_n^1)|\alpha\right]\\
&\leq W(\alpha, \delta_\lambda; A_n). 
\end{align*}
Applying the induction hypothesis, and then (\ref{Wdef1}) and (\ref{Wdef3}), we get 
\begin{align*}
\phi_1 &\leq \phi_2 \leq a_1\lambda + W(\alpha, \delta_\lambda; A_n^1) \leq W(\alpha, \delta_\lambda; A_n). 
\end{align*}
That is, $\phi_i \leq W(\alpha, \delta_\lambda; A_n)$ for $i=0, 1, 2.$  Hence the claim holds by (\ref{lam_mono}). 
\end{proof}
\begin{proof}[Proof of Proposition~\ref{prop1}]
Suppose $\lambda=\Lambda(\alpha; A_n)$.  By the optimal stopping characterization, we have 
$W(\alpha, \delta_\lambda; A_n^1) = \lambda \sum_{i=2}^n a_i.$  Lemma~\ref{lem3} yields $W(\alpha +\delta_\lambda, \delta_\lambda; A_n^1)\leq \lambda \sum_{i=2}^n a_i$.  It follows from Lemma~\ref{lemlam} that $\lambda\geq \Lambda(\alpha +\delta_\lambda; A_n^1)$.  That is, $\Lambda(\alpha; A_n)\geq \Lambda(\alpha +\delta_\lambda; A_n^1)$, which implies $\lambda\leq b(\alpha; A_n)$ (under the assumptions $b(\alpha; A_n)$ is unique). 
\end{proof}

\section{Prior weight monotonicity}
The main result of this section (Theorem \ref{thm2}) shows that the maximum expected payoff of a bandit decreases as the prior weight for the Dirichlet process prior of an arm increases.  When arm 2 is known and the discount sequence is regular, this shows that the break-even value $\Lambda(M_1F_1; A_n)$ decreases as $M_1$ (the prior weight associated with arm 1) increases.  That is, given the same immediate payoff, arm 1 becomes less desirable as the amount of information about it increases. 

\begin{theorem}
\label{thm2}
Let $F$ be a probability distribution on $\mathbf{R}$ with a finite mean.  If $0<M< \tilde{M}$ then 
\begin{equation}
\label{mono2}
W(M F, \alpha_2; A_n)\geq W(\tilde{M} F, \alpha_2; A_n).
\end{equation}
\end{theorem}

Lemma \ref{lemlam} and Theorem \ref{thm2} yield the following result concerning the break-even value $\Lambda(\alpha; A_n)$ for the one armed bandit $(\alpha, \delta_\lambda; A_n)$, as conjectured by Clayton and Berry (1985) in the case of uniform discounting. 

\begin{corollary}
\label{coro2}
For $0<M< \tilde{M}$ we have $\Lambda(MF; A_n)\geq \Lambda(\tilde{M} F; A_n)$, assuming $A_n$ is a regular discount sequence. 
\end{corollary}

When $F$ has only two support points, Corollary 2 says that for a Bernoulli one-armed bandit with a ${\rm Beta}(Mu, Mv)$ prior, $u, v>0$, for the unknown arm, the break-even value decreases in $M$.  This Bernoulli case was proved by Gittins and Wang (1992) for infinite-horizon geometric discounting. 

The rest of this section gives a proof of Theorem \ref{thm2}.  We assume $F$ has finite, and then bounded, and finally arbitrary, support.  The key step is summarized as Lemma \ref{lem4}.

\begin{lemma}
\label{lem4} 
Assume $n\geq 2,\ L>0$.  Assume $\alpha$ is a finite measure on $\mathbf{R}$ with a finite mean and $F$ is a probability distribution on $\mathbf{R}$ with $s<\infty$ support points.  Then $E[W(\alpha + \theta F + (L-\theta) \delta_X, \alpha_2; A^1_n)|F]$ decreases in $\theta \in [0, L]$. 
\end{lemma}
\begin{proof}
We use induction on $s$.  Although the induction may start at the trivial case $s=1$, we present the $s=2$ case to illustrate the convexity arguments.  Write $F=p \delta_1 + (1-p) \delta_0$ where $p\in (0, 1)$ and $\{0, 1\}$ are the support points without loss of generality.  For fixed $0 \leq \theta_1< \theta_2\leq L$, let $Z\sim {\rm Bernoulli}(p)$ and define 
$$Z_i=\theta_i p + (L-\theta_i)Z,\quad i=1,2.$$
Then $E Z_1 = E Z_2 = p L$, and it is easy to verify $Z_2\leq_{\rm cx} Z_1$ as $\theta_1< \theta_2$ (see, e.g., Shaked and Shanthikumar 2007, Theorem 3.A.18).  Let us define 
$$\phi(u) = W(\alpha + u\delta_1 +(L-u) \delta_0, \alpha_2; A^1_n).$$  
By direct calculation
\begin{align*}
E\left[W\left(\alpha + \theta_1 F +(L-\theta_1)\delta_X, \alpha_2; A^1_n\right)|F\right] = &p \phi(\theta_1 p +L-\theta_1) + (1-p) \phi( \theta_1 p)\\ 
=& E \phi(Z_1)\\
\geq & E \phi(Z_2)\\
=& E\left[ W\left(\alpha + \theta_2 F +(L-\theta_2)\delta_X, \alpha_2; A^1_n \right)|F\right]
\end{align*}
where the inequality holds because $Z_2\leq_{\rm cx} Z_1$ and, by Lemma \ref{lem1}, $\phi(u)$ is convex in $u\in [0, L]$.  

For $s\geq 3$, write $F=\sum_{j=1}^s p_j \delta_{x_j},$ where $\{x_j,\ j=1,\ldots, s\}$ are the support points, $p_j> 0$ and $\sum_{j=1}^s p_j =1$.  Consider the leave-one-out distributions 
\begin{align*}
F^k &=\sum_{j\neq k} \frac{p_j }{1-p_k} \delta_{x_j},\quad k=1,\ldots, s.
\end{align*}
Denote $W(\gamma) = W(\gamma, \alpha_2; A^1_n)$ for convenience.  For fixed $0 \leq \theta_1< \theta_2\leq L$, we have
\begin{align}
\nonumber
(s-1) &\left. E\left[W\left(\alpha  +\theta_1 F  + (L-\theta_1) \delta_X\right)\right|F\right]\\
\nonumber
& =\sum_{k=1}^s (1-p_k) E\left[W\left(\alpha + \theta_1 F + (L-\theta_1) \delta_X\right)\left|F^k \right]\right.\\
\nonumber
&=\sum_{k=1}^s (1-p_k) \left. E\left[W\left(\alpha + \theta_1 p_k \delta_{x_k}+ \theta_1 (1-p_k) F^k + (L-\theta_1) \delta_X\right)\right|F^k \right]\\
\label{induct1}
&\geq \sum_{k=1}^s (1-p_k)\left. E\left[W\left(\alpha + \theta_1 p_k \delta_{x_k} + \theta_2 (1-p_k) F^k + (L -\theta_2(1-p_k)-\theta_1 p_k) \delta_X\right)\right|F^k \right]\\
\nonumber
&= \sum_{k=1}^s \sum_{j\neq k} p_j V_{jk}, 
\end{align}
where 
\begin{align*}
V_{jk} &=W\left(\alpha +\theta_2 \gamma^{jk} + \theta_1 p_k \delta_{x_k} + (L -\theta_2(1-p_k - p_j)-\theta_1 p_k) \delta_{x_j}\right),\\
\gamma^{jk} &=\sum_{l\neq j, k} p_l\delta_l,\quad j\neq k.
\end{align*}
The inequality (\ref{induct1}) follows from the induction hypothesis; other steps are algebraic manipulations.

For fixed $j\neq k$, let $Z\sim {\rm Bernoulli} (p_k/(p_j + p_k))$ and define 
\begin{align*}
Z_1 &= \theta_1 p_k + Z(L-\theta_2+(\theta_2-\theta_1)(p_j+p_k));\\
Z_2 &=\theta_2 p_k + Z(L-\theta_2). 
\end{align*}
It is easy to verify that
$$EZ_1=EZ_2;\quad Z_2\leq_{\rm cx} Z_1.$$ 
We have
\begin{align*}
p_j V_{jk} + p_k V_{kj} &= (p_j+p_k) EW\left(\alpha +\theta_2 \gamma^{jk} + Z_1 \delta_{x_k} + (L-\theta_2(1-p_k-p_j) -Z_1)\delta_{x_j}\right)\\
&\geq (p_j+p_k) EW\left(\alpha +\theta_2 \gamma^{jk} + Z_2 \delta_{x_k} + (L-\theta_2(1-p_k-p_j) -Z_2)\delta_{x_j}\right)\\
&= p_j W\left(\alpha +\theta_2 F + (L-\theta_2) \delta_{x_j}\right) + p_k W\left(\alpha + \theta_2 F + (L -\theta_2) \delta_{x_k}\right),
\end{align*}
where the inequality holds by Lemma \ref{lem1} as $Z_2\leq_{\rm cx} Z_1$.  Hence, 
\begin{align*}
\sum_{k=1}^s \sum_{j\neq k} p_j V_{jk} &= \sum_{1\leq j<k\leq s} (p_j V_{jk} + p_k V_{kj})\\
&\geq \sum_{1\leq j<k\leq s} \left[p_j W\left(\alpha +\theta_2 F + (L-\theta_2) \delta_{x_j}\right) + p_k W\left(\alpha + \theta_2 F + (L -\theta_2) \delta_{x_k}\right) \right]\\
&=(s-1)\sum_{j=1}^s p_j W\left(\alpha +\theta_2 F + (L-\theta_2) \delta_{x_j}\right)\\
&=(s-1) E[W(\alpha   +\theta_2 F  + (L-\theta_2) \delta_X)|F]. 
\end{align*} 
Thus we have shown that $E[W(\alpha   +\theta F  + (L-\theta) \delta_X)|F]$ decreases in $\theta\in [0, L]$. 
\end{proof}

\begin{proof}[Proof of Theorem \ref{thm2}]
(i) Assume $F$ has finite support.  The claim obviously holds for $n=1$.  For $n\geq 2$ we use induction.  In view of (\ref{Wdef1})--(\ref{Wdef3}), we only need to show 
\begin{align}
\label{wt1}
E\left[W(M F +\delta_X, \alpha_2; A^1_n)|F\right] &\geq E\left[W(\tilde{M} F +\delta_X, \alpha_2; A^1_n)|F\right]\quad {\rm and}\\
\label{wt2}
E\left[W(M F, \alpha_2 + \delta_Y; A^1_n)|\alpha_2\right] &\geq E\left[W(\tilde{M} F, \alpha_2 +\delta_Y; A^1_n)|\alpha_2\right].
\end{align}
By the induction hypothesis, (\ref{wt2}) holds.  Define $\eta=(\tilde{M}+1)/(M+1)$ and $\theta=\tilde{M}/\eta$.  Noting $M< \theta< M+1$, we may apply Lemma \ref{lem4} and get 
\begin{align}
\nonumber
E\left[W(M F +\delta_X, \alpha_2; A^1_n)|F\right] \geq & E\left[W(\theta F + (M+1-\theta)\delta_X, \alpha_2; A^1_n)|F\right] \\
\label{strict}
\geq & E\left[W(\eta (\theta F + (M+1-\theta)\delta_X), \alpha_2; A^1_n)|F\right]\\
\nonumber
= & E\left[W(\tilde{M} F +\delta_X, \alpha_2; A^1_n)|F\right],
\end{align}
where (\ref{strict}) holds by the induction hypothesis, as $\eta> 1$.  Thus (\ref{wt1}) holds as required.

(ii) Assume $F$ has bounded support.  Then for arbitrary $\epsilon>0$ we can construct two distributions $F^*$ and $F_*$ supported on $\{x_1,\ldots, x_s\}$ and $\{x_0, \ldots, x_{s-1}\}$ respectively, where $x_j=x_0+j\epsilon,$ such that $F(x_0)=0,\ F(x_s)=1$ and $F_*(x_j)=F^*(x_{j-1})=F(x_j),\ j=1,\ldots, s$.  By construction, $F_*\leq_{\rm st} F\leq_{\rm st} F^*$.  Theorem \ref{thm1} yields 
$$W(MF_*, \alpha_2; A_n)\leq W(MF, \alpha_2; A_n)\leq W(MF^*, \alpha_2; A_n).$$
Note that if $X\sim F^*$ then $X-\epsilon\sim F_*$.  Therefore the bandits $(MF^*, \alpha_2; A_n)$ and $(MF_*, \alpha_2; A_n)$ can be coupled in an obvious way such that, for every strategy of $(MF^*, \alpha_2; A_n)$, there exists a strategy of $(MF_*, \alpha_2; A_n)$ under which the payoff at each stage is either the same (when arm 2 is selected), or exactly $\epsilon$ less (when arm 1 is selected).  Thus we have shown 
$$W(MF^*, \alpha_2; A_n) - W(MF_*, \alpha_2; A_n)\leq \epsilon \sum_{i=1}^n a_i.$$
Hence $W(MF^*, \alpha_2; A_n) \to W(MF, \alpha_2; A_n)$ as $\epsilon\to 0,$ and the monotonicity of $W(MF^*, \alpha_2; A_n)$ with respect to $M$ implies the corresponding monotonicity of $W(MF, \alpha_2; A_n)$. 

(iii) Finally, assume $F$ is an arbitrary distribution with a finite mean.  Suppose $X\sim F$.  For $L>0$ let $F^*$ be the distribution of $X^*$, defined as $X$ if $|X|\leq L$ and $0$ otherwise.  We construct a coupling between $(MF, \alpha_2; A_n)$ and $(MF^*, \alpha_2; A_n)$.  Let $X_k$ be the resulting observation when arm 1 of $(MF, \alpha_2; A_n)$ is pulled for the $k$th time.  If $|X_1|\leq L$ then let $X_1^*=X_1$, otherwise $X_1^*=0$, yielding $X_1^*\sim F^*$.  For general $k\geq 1$, if $|X_i|\leq L,\ i=1,\ldots, k,$ then let $X_{k+1}^*=X_{k+1}$ if $|X_{k+1}|\leq L$ and $X_{k+1}^*=0$ otherwise.  In this case the conditional distribution of $X_{k+1}$ given $X_i,\ i=1,\ldots, k,$ is $(MF+\sum_{i=1}^k \delta_{X_i})/(M+k)$.  Since $|X_i|\leq L,\ i=1,\ldots, k,$ we have $X_i^*=X_i,\ i=1,\ldots, k$, and the conditional distribution of $X_{k+1}^*$ given $X_i^*,\ i=1,\ldots, k,$ is precisely $(MF^*+\sum_{i=1}^k \delta_{X_i^*})/(M+k)$.  That is, $X_i^*,\ i=1,\ldots, k+1,$ can be regarded as successive pulls from arm 1 of $(MF^*, \alpha_2; A_n)$ as long as $|X_i|\leq L,\ i=1,\ldots, k$.  Let the $k$th pull from arm 2 be $Y_k$ for both bandits.  In the event that all $|X_i|\leq L,\ i=1,\ldots, n$, the optimal strategy for $(MF, \alpha_2; A_2)$ can be adopted for $(MF^*, \alpha_2; A_2)$ throughout, yielding identical pulls (not all $X_i,\ i=1,\ldots,n,$ are realized).  By considering a trivial upper (respectively, lower) bound for the payoff of $(MF, \alpha_2; A_2)$ (respectively, $(MF^*, \alpha_2; A_2)$) when at least one $|X_i|>L$, we have 
\begin{align*}
W(MF, \alpha_2; A_2) - W(MF^*, \alpha_2; A_2) &\leq E\left[1_{\cup_{i=1}^n \{|X_i|>L\}} \sum_{i=1}^n \left(a_i(|Y_i| +|X_i|) - a_i(-|Y_i| -L)\right)\right]\\
&\leq E\left[1_{\cup_{i=1}^n \{|X_i|>L\}} \sum_{i=1}^n a^*(2|Y_i| +|X_i|+L)\right]\\
&\leq E\left[\left(\sum_{i=1}^n 1_{\{|X_i|>L\}}\right) \sum_{i=1}^n a^*(2|Y_i| +|X_i|+L)\right]\\
&\equiv a^* h(L), 
\end{align*}
where $a^*\equiv \max_{i=1}^n a_i$.  Direct calculation using exchangeability yields 
\begin{align*}
h(L)=n^2 \Pr(|X_1|>L) (2E|Y_1| + L) + n E\left[1_{|X_1|>L} |X_1|\right] + n(n-1) E \left[1_{|X_1|>L} |X_2|\right]
\end{align*}
The first two terms tend to zero as $L\to\infty$ by dominated convergence since $E|X_1|<\infty$.  For the last term, by conditioning on $X_1$ we have
$$E \left[1_{|X_1|>L} |X_2|\right] = E \left[1_{|X_1|>L} \left(\frac{M}{M+1} E|X| + \frac{1}{M+1} |X_1|\right)\right],$$
which also vanishes as $L\to\infty$.  Thus 
$$\limsup_{L\to\infty} \left[W(MF, \alpha_2; A_2) - W(MF^*, \alpha_2; A_2)\right]\leq  0.$$
By a parallel argument, we get $\liminf_{L\to\infty} \left[W(MF, \alpha_2; A_2) - W(MF^*, \alpha_2; A_2)\right]\geq  0.$  Thus $W(MF^*, \alpha_2; A_2)$ tends to $W(MF, \alpha_2; A_2)$ as $L\to\infty$, and the monotonicity of $W(MF, \alpha_2; A_n)$ with respect to $M$ is proved as before. 
\end{proof}

{\bf Remark 2.} Clayton and Berry (1985) also conjecture that the monotonicity in Corollary \ref{thm2} is strict if $n\geq 2,\ A_n=(1,1, \ldots, 1)$, and $F$ is nondegenerate.  This can be confirmed by a careful analysis of the above results.  Some modifications are needed.  Using arguments similar to steps (ii) and (iii) in the proof of Theorem \ref{thm2}, we can first establish that Lemma \ref{lem4} holds without the finite support restriction.  Directly applying this strengthened Lemma \ref{lem4} shows that (\ref{mono2}) holds with strict inequality assuming $n\geq 2,\ A_n=(1,1, \ldots, 1),\ F$ is nondegenerate, and arm 1 is optimal initially in $(\tilde{M}F, \alpha_2; A_n)$.  Under such conditions, the strictness of the inequality holds by induction as one key step (\ref{strict}) holds with strict inequality.  It follows that Corollary \ref{coro2} can be strengthened to strict monotonicity assuming uniform discounting, $n\geq 2,$ and a nondegenerate $F$. 

%\section*{Acknowledgement}


\begin{thebibliography}{10}

\bibitem{B72}
D. A. Berry, A Bernoulli two-armed bandit, {\it Ann. Math. Statist.} {\bf 43} (1972) 871--897. 

\bibitem{BF79}
D. A. Berry and B. Fristedt, Bernoulli one-armed bandits---arbitrary discount sequences, {\it Ann. Statist.} {\bf 7} (1979) 1086--1105. 

\bibitem{BF86}
D. A. Berry and B. Fristedt, {\it Bandit Problems: Sequential Allocation of Experiments} (1985) Chapman and Hall, New York.

\bibitem{BJK56}
R. N. Bradt, S. M. Johnson and S. Karlin, On sequential designs for maximizing the sum of $n$ observations, {\it Ann. Math. Statist.} {\bf 27} (1956) 1060--1074. 

\bibitem{Ch94}
M. K. Chattopadhyay, Two-armed Dirichlet bandits with discounting, {\it Ann. Statist.} {\bf 22} (1994) 1212--1221. 

\bibitem{C68}
H. Chernoff, Optimal stochastic control, {\it Sankhya A} {\bf 30} (1968) 221--252.

\bibitem{CB85}
M. K. Clayton and D. A. Berry, Bayesian nonparametric bandits, {\it Ann. Statist.} {\bf 13} (1985) 1523--1534. 

\bibitem{F73}
T. S. Ferguson, A Bayesian analysis of some nonparametric problems, {\it Ann. Statist.} {\bf 1} (1973) 209--230. 

\bibitem{G79}
J. C. Gittins, Bandit processes and dynamic allocation indices (with discussion), {\it Journal of the Royal Statistical Society, Series B} {\bf 41} (1979) 148--177. 

\bibitem{GJ74}
J. C. Gittins and D. M. Jones, A dynamic allocation index for the sequential design of experiments. In: J. Gani, Editor, {\it Progress in Statistics,} North-Holland, Amsterdam (1974) 241–-266. 

\bibitem{GW92}
J. C. Gittins and Y.-G. Wang, The learning component of dynamic allocation indices, {\it Ann. Statist.} {\bf 20} (1992) 1625--1636. 

\bibitem{H92}
S. J. Herschkorn, Bandit bounds from stochastic variability extrema, {\it Stat. Prob. Lett.} {\bf 35} (1997) 283--288.

\bibitem{MS02}
A. M\"{u}ller and D. Stoyan, {\it Comparison Methods for Stochastic Models
and Risks}, Wiley \& Sons, Chichester (2002). 

\bibitem{SS07}
M. Shaked and J. G. Shanthikumar, {\it Stochastic Orders}, Springer, New York (2007). 

\bibitem{W80}
P. Whittle, Multi-armed bandits and the Gittins index, {\it J. Roy. Statist. Soc. B} {\bf 42} (1980) 143–-149. 

\end{thebibliography}
\end{document}